\documentclass[psamsfonts,leqno]{amsart}
  \usepackage{amsthm}
  \usepackage{amsmath}
  \usepackage{amssymb}
  \usepackage{mathrsfs}
  \newcommand{\uhr}{\upharpoonright}
  \newtheorem{theorem}{Theorem}[section]
  \newtheorem{corollary}[theorem]{Corollary}
  \newtheorem{proposition}[theorem]{Proposition}
  
  \newtheorem{question}[theorem]{Question}
  
  \theoremstyle{definition}
  \newtheorem{definition}[theorem]{Definition}
  \newtheorem{remark}[theorem]{Remark}
  \newtheorem{example}[theorem]{Example}

  \numberwithin{equation}{section}

  \title[Bornoligies, Topological Games and Function Spaces]
  {Bornoligies, Topological Games and\\ Function Spaces}
  \author[J. Cao]{Jiling Cao}
  \address{School of Computer and Mathematical Sciences,
  Auckland University of Technology, Private Bag 92006, Auckland
  1142, New Zealand}
  \email{jiling.cao@aut.ac.nz}
  \author[A.~H. Tomita]{Artur H. Tomita}
  \address{Departamento de Matem\'atica, Instituto de
  Matem\'atica e Estat\'{i}stica, Universidade de S\~ao Paulo,
  Rua do Mat\~ao 1010, Cep 05508-090, S\~ao Paulo, Brasil}
  \email{tomita@ime.usp.br}
  \thanks{\hspace{-1.66em} 2010 \emph{Mathematics Subject
  Classification.}
  Primary 54C35; Secondary 46A17, 54E50, 54E52.}

  \date{}

  \dedicatory{}

\begin{document}

  \begin{abstract}
  In this paper, we continue the study of function spaces equipped
  with topologies of (strong) uniform convergence on bornologies
  initiated by Beer and Levi \cite{beer-levi:09}. In particular,
  we investigate some topological properties these function spaces
  defined by topological games. In addition, we also give further
  characterizations of metrizability and completeness properties
  of these function spaces.
  \end{abstract}

  \maketitle

  \section{Introduction} \label{sec:intro}

  A bornology $\mathscr B$ on a (nonempty) set $X$ is a family of
  nonempty subsets of $X$ that is closed under taking finite unions,
  that is closed under taking subsets, and that forms a cover of
  $X$. In his pioneering work \cite{hu:49}, Hu investigated
  bornologies defined on metrizable spaces that correspond to that
  of bounded sets with respect to an admissible metric. He was the
  first to build a framework to discuss macroscopic phenomena and
  their interplay with the local structures investigated in general
  topology. In the literature, bornologies have been widely applied
  in the theory of locally convex spaces \cite{hogbe:77}, where
  additional conditions are required, e.g., that the bornology be
  closed under vector addition and scalar multiplication, and
  perhaps that the balanced hull of $B\in {\mathscr B}$ remains
  in $\mathscr B$. Furthermore, over the past 15 years, Borwein
  et al. \cite{borwein:93} have used bornologies to develop a
  unified theory of differentiability for functions defined on
  normed spaces, characterizing various geometric properties
  of spaces in terms of the relationship between different
  bornological derivatives on various classes of functions.

  \medskip
  Recently, there has been renewed interest in bornologies in
  analysis and topology. On one hand, Beer and Levi
  \cite{beer-levi:09b} investigated some basic questions
  about the totally bounded subsets induced by a bornology
  $\mathscr B$ on a metric space. On the other hand, Beer and
  Levi \cite{beer-levi:09} applied bornologies to study
  function spaces with topologies of uniform convergence on
  bornologies in a unified approach. In the same paper, they
  also introduced topologies of strong uniform convergence
  on bornologies. This study has been carried on by Caserta
  et al. \cite{caserta-dimaio-hola:10} and Hol\'{a}
  \cite{hola:12}. In \cite{caserta-dimaio-hola:10}, the authors
  characterized several topological properties of the topology
  of strong uniform convergence on a bornology, including
  separability, second countability, countable tightness, etc.
  In \cite{hola:12}, Hol\'{a} characterized complete
  metrizability of topologies of strong uniform convergence
  on bornologies. The purpose of this article is to continue
  the study of topological properties of function space
  equipped with topologies of (strong) uniform convergence on
  bornologies. In particular, we investigate those properties
  defined in terms of topological games.

  \medskip
  Given a topological space $X$ and a metric space
  $(Y,\rho)$, let $C(X,Y)$ be the set of continuous functions
  from $X$ to $(Y, \rho)$. For $(Y,\rho) = ({\mathbb R}, \mbox{
  Euclidean metric})$, $C(X,Y)$ is simply denoted by $C(X)$.
  The family of non-empty finite subsets (resp. subsets with
  compact closure) of $X$ is denoted by ${\mathscr F}(X)$ (resp.
  ${\mathscr K}(X)$). In addition, the family of non-empty
  subsets of $X$ is denoted by ${\mathscr P}(X)$. Two nonempty
  subsets $A$ and $B$ of $(Y,\rho)$ are said to be \emph{far
  from each other} if
  \[
  D(A,B):= \inf \{\rho(a,b): a\in A \mbox{ and } b\in B\}
  >0.
  \]
  Given $\delta >0$, the open $\delta$-ball centered at $a\in Y$
  is denoted by $S(a, \delta)$, and $A^\delta := \bigcup_{a\in A}
  S(a,\delta)$ is called the \emph{$\delta$-enlargement} of the
  subset $A$ of $(Y,\rho)$. If $\mathscr A$ is a family of
  nonempty subsets of $(Y,\rho)$, then ${\mathscr A}^\delta :=
  \{A^\delta: A \in {\mathscr A} \}$. Given a function $f\in C(X,Y)$
  and a net $\langle f_\alpha: \alpha \in D \rangle$ pointwise
  convergent to $f$, the family of subsets of $X$ on which the
  convergence is uniform is easily seen to be a bornology with
  closed base. Here, a subfamily $\mathscr S$ of a bornology
  $\mathscr B$ is a \emph{base} of $\mathscr B$ if for each member
  $B \in \mathscr B$ there exists an $A \in \mathscr S$ such that
  $B \subseteq A$. If each member of a base $\mathscr S$ is a
  closed subset, then $\mathscr S$ is called a \emph{closed base}
  of $\mathscr B$. Bearing this in mind, suppose that
  $\mathscr B$ is a bornology with closed base on $X$. The
  uniformity $\Delta_{\mathscr B}$ for the topology ${\mathscr
  T}_{\mathscr B}$ of uniform convergence on $\mathscr B$ for
  $C(X,Y)$ has as a base for its encourages all sets of the form
  \[
  [B; \varepsilon] :=\{(f,g): \rho(f(x),
  g(x)) <\varepsilon \mbox{ for all } x\in B\},
  \]
  where $B\in {\mathscr B}$ and $\varepsilon >0$. When ${\mathscr B}
  ={\mathscr F}(X)$, we get the standard uniformity for the topology
  of pointwise convergence. The topology ${\mathscr T}_{{\mathscr
  F}(X)}$ is simply denoted by ${\mathscr T}_p$, and the space
  $(C(X, Y), {\mathscr T}_{{\mathscr F}(X)})$ is simply denoted by
  $C_p(X,Y)$. When ${\mathscr B} = {\mathscr K}(X)$, we get the
  standard uniformity for the topology of uniform convergence on
  compacta. The topology ${\mathscr T}_{{\mathscr K}(X)}$ is simply
  denoted by ${\mathscr T}_k$, and the space $(C(X, Y), {\mathscr
  T}_{{\mathscr K}(X)})$ is simply denoted by $C_k(X,Y)$. Finally,
  when ${\mathscr B} ={\mathscr P}(X)$, we get the standard
  uniformity for the topology of uniform convergence on $X$.

  \begin{definition}
  Let $(X,d)$ and $(Y,\rho)$ be two metric spaces, and let $\mathscr
  B$ be a bornology with a closed base on $X$. The \emph{topology
  of strong uniform convergence} on $\mathscr B$ for $C(X,Y)$,
  denoted by ${\mathscr T}_{\mathscr B}^s$, is determined by a
  uniformity $\Delta_{\mathscr B}^s$, which has as a base for its
  entourages all sets of the form
  \[
  [B; \varepsilon]^s :=\{(f,g): \mbox{ there is } \delta >0 \mbox{
  such that } \rho(f(x),g(x)) <\varepsilon \mbox{ for all } x\in
  B^\delta\},
  \]
  where $B\in {\mathscr B}$ and $\varepsilon >0$.
  \end{definition}

  Note that ${\mathscr T}_p \subseteq {\mathscr T}_{\mathscr B}$
  on $C(X,Y)$ for any bornology $\mathscr B$ on $X$. Thus, $(C(X),
  {\mathscr T}_{\mathscr B})$ and $(C(X),{\mathscr T}_{\mathscr
  B}^s)$ are Tychonoff topological groups. Since $[B;\varepsilon]^s
  \subseteq [B;\varepsilon]$ for all $\mathscr B$ and
  $\varepsilon >0$, ${\mathscr T}_{\mathscr B}^s$ is always finer
  than ${\mathscr T}_{\mathscr B}$ on $C(X,Y)$. In
  general, these two topologies are distinct on $C(X,Y)$. Several
  conditions under which ${\mathscr T}_{\mathscr B}^s= {\mathscr
  T}_{\mathscr B}$ on $C(X,Y)$ can be found in the literature,
  e.g., in \cite{beer-costantini:13}, \cite{beer-levi:09} and \cite{caserta-dimaio-hola:10}.

  \medskip
  The paper is organized as follows. In Section \ref{sec:games}, we
  introduce and study several classes of topological spaces defined
  by topological games. These classes of spaces are close to the
  spaces with certain completeness properties. In Section
  \ref{sec:metrizable}, we consider the metrizability of topologies
  ${\mathscr T}_{\mathscr B}$ and ${\mathscr T}_{\mathscr B}^s$
  on $C(X,Y)$. Further characterizations of these topologies in
  terms of $q_D$-points are given. In Section \ref{sec:complete},
  we apply results in Section \ref{sec:metrizable} to study
  completeness properties of topologies ${\mathscr T}_{\mathscr B}$
  and ${\mathscr T}_{\mathscr B}^s$ on $C(X,Y)$. The last section
  is devoted to the study of the Baire property of function spaces
  defined by bornologies. We give some necessary conditions under
  which a function space equipped with the topology of strong
  uniform topology on a bornology to be Baire. The paper also
  contains several open questions. All undefined concepts can be
  found in the listed references, e.g., \cite{en:89}, \cite{Gr:84}
  or \cite{mccoy-natantu:88}.

  \section{Topological spaces defined by games}
  \label{sec:games}
  Let $T$ be a topological space. Recall that the \emph{Choquet
  game} $Ch(T)$ played in $T$ is the following two-player infinite
  game. Players, called $\beta$ (the first player) and $\alpha$
  (the second player), alternatively choose nonempty open subsets
  of $X$ with $\beta$ starting first such that $U_1 \supseteq V_1
  \supseteq U_2 \supseteq V_2 \supseteq \cdots$.
  In this way, a run $\langle (U_n, V_n): n\ge 1\rangle$ will be
  produced. Then $\alpha$ is said to {\em win\/} this run if
  $\bigcap_{n \ge 1} U_n (=\bigcap_{n\ge 1} V_n)\not= \emptyset$.
  Otherwise, we say that $\beta$ has won. By a \emph{strategy}
  $\sigma$ for player $\beta$, we mean a function whose domain
  is the collection of all legal finite sequences of moves of
  $\alpha$ and whose values are nonempty open sets of $X$. If
  $\sigma$ is a strategy for $\beta$ in $Ch(T)$, $\sigma(\emptyset)$
  denotes the first move of $\beta$. A finite sequence $\langle
  V_1,...,V_n\rangle$ of nonempty open sets of $X$ is called a
  \emph{partial play} of $\alpha$ subject to $\sigma$ in $Ch(X)$
  if $V_1\subseteq \sigma (\emptyset)$ and $V_{i+1} \subseteq
  \sigma (V_1,...,V_i) \subseteq V_i$ for all $1\le i <n$. Similarly,
  an infinite sequence $\langle V_n: n\ge 1\rangle$ of nonempty
  open sets of $X$ is called a \emph{(full) play} of $\alpha$
  subject to $\sigma$ if $V_1 \subseteq \sigma (\emptyset)$ and
  $V_{n+1} \subseteq \sigma (V_1,...,V_n) \subseteq V_n$ for all
  $n \ge 1$. Strategies for player $\alpha$, partial plays and
  (full) plays for $\beta$ subject to a strategy of $\alpha$ can
  be defined similarly. In addition, a \emph{winning strategy}
  for a player is a strategy such that this player wins each
  play of its opponent subject to this strategy no matter how
  the opponent moves in the game. If $\alpha$ has a winning
  strategy in $Ch(X)$, then $X$ is called \emph{weakly
  $\alpha$-favorable}. For more information on topological
  games, see \cite{cao-moors:06}.
  
  \medskip
  Recall that $T$ is \textit{Baire} if the intersection of every
  sequence of dense open subsets in $T$ is dense. The well-known
  Baire Category Theorem in analysis and topology claims that every
  complete metric or locally compact Hausdorff space is Baire. A
  Tychonoff space $T$ is \emph{\v{C}ech complete} if it is a
  $G_\delta$-set in its \v{C}ech-Stone compactification $\beta T$.
  Since the class of \v{C}ech complete spaces contains all
  complete metric spaces and locally compact topological spaces,
  a unified version of the Baire Category Theorem is: \emph{Each
  \v{C}ech complete space is Baire}. An interesting connection
  between the Choquet game and the class of Baire spaces is given
  in the following theorem of Oxtoby in 1957.

  \begin{theorem}[\cite{oxtoby:57}]\label{thm:krom}
  A space $X$ is Baire if and only if the first player does not
  have a winning strategy in the Choquet game played in $X$.
  \end{theorem}

  A proof of Theorem \ref{thm:krom} can also be found in
  \cite{krom:74} and \cite{saint-raymond:83}. It follows
  from Theorem \ref{thm:krom}  that every weakly $\alpha$-favorable
  space is Baire. Let $D \subseteq T$ be a dense subset of $T$.
  The game ${\mathscr G}_S(D)$, involves two players $\alpha$ and
  $\beta$. Players $\beta$ and $\alpha$ choose alternately non-empty
  open subsets $V_n$ and $U_n$ in $X$ just as in the Choquet game.
  Player $\alpha$ \emph{wins} a run if $\bigcap_{n\ge 1} U_n$
  is non-empty and each sequence $\langle x_n: n \ge 1\rangle$ with
  $x_n \in U_n \cap D$ for all $n\ge 1$ has a cluster point in $T$.
  The space $T$ is \emph{strongly Baire} \cite{KKM:01} if $T$ is
  regular and there is a dense subset $D \subset T$ such that
  $\beta$ does not have a winning strategy in the game ${\mathscr
  G}_S(D)$ played in $T$.

  \medskip
  A family of subsets of $T$ is said to be an \emph{almost cover}
  if the union of its members is dense in $T$. A {\em sieve\/}
  (resp. An \emph{almost sieve}) $\langle \{ U_i: i \in I_n \},
  \pi_n: n \ge 1\rangle$ on $T$ is a sequence of indexed covers
  (resp. almost covers) $\{ U_i: i \in I_n \}$ of $X$, together
  with maps $\pi_n: I_{n+1}\to I_n$ such that $U_i =X$ for all $i
  \in I_0$ and $U_i=\bigcup \{ U_j: j\in \pi_n^{-1}(i)\}$ (resp.
  $\bigcup \{ U_j: j\in \pi_n^{-1}(i)\}$ is dense in $U_i$) for all
  $i\in I_n$ and all $n\ge 1$. Moreover, a {\em $\pi$-chain\/} for
  such a sieve or an almost sieve is a sequence $\langle i_n: n
  \ge 1\rangle$ such that $i_n \in I_n$ and $\pi_n(i_{n+1})=i_n$
  for all $n \ge 1$. A filterbase ${\mathscr F}$ on $T$ is said to be
  \emph{controlled} by a sequence $\langle U_n: n \ge 1\rangle$ of
  subsets of $X$ if each $U_n$ contains some $F\in {\mathscr F}$.
  If each filterbase controlled by $\langle U_n: n \ge 1 \rangle$
  clusters, then $\langle U_n: n \ge 1 \rangle$ is called a
  \emph{complete sequence} on $T$. Furthermore, $T$ is called
  \emph{sieve complete} (resp. \emph{almost complete}) \cite{CCN:74},
  \cite{michael:91} if there exists an (resp. almost) open sieve
  such that $\langle U_{i_n}: n \ge 1 \rangle$ is a complete
  sequence for each $\pi$-chain $\langle i_n: n \ge 1  \rangle$,
  and $T$ is called an \emph{almost monotonically $p$-space} if
  $T$ has an almost open sieve such that $\langle U_{i_n}: n \ge
  1 \rangle$ is a complete sequence for each $\pi$-chain $\langle
  i_n: n\ge 1 \rangle$ with $\bigcap_{n \ge 1} U_{i_n}\ne \emptyset$.

  \medskip
  All \v{C}ech complete spaces are sieve complete, all sieve
  complete spaces are almost complete, all almost complete spaces
  are strongly Baire, and all strongly Baire spaces are Baire.
  Each almost complete Tychonoff space has a \v{C}ech complete
  dense $G_\delta$-subspace. In \cite{KKM:01}, some conditions
  under which a Baire space is strongly Baire were given. Below,
  we give some more conditions.

  \begin{theorem}
  Let $T$ be a regular space. If $T$ is a Baire and almost
  monotonically $p$-space, then $T$ is a strongly Baire space.
  \end{theorem}

  \begin{proof}
  Let $\langle \{ U_i: i \in I_n \},\pi_n: n \ge 1\rangle$ be an
  almost open sieve on $T$ to witness $T$ being an almost
  monotonically $p$-space. For each $n\ge 1$, put $G_n =
  \bigcup_{i\in I_n} U_i$. By definition, each $G_n$ is a dense
  open subset of $T$. Since $T$ is a Baire space, then $D=
  \bigcap_{n\ge 1} G_n$ is also a dense subset of $T$.

  \medskip
  We shall show that any strategy $\sigma$ for player $\beta$ in
  the game ${\mathscr G}_S(D)$ is not a winning strategy for $\beta$.
  To this end, we need to define inductively a strategy $\theta$ for
  $\beta$ in $Ch(T)$ by applying $\sigma$ and then applying the
  Baireness of $T$. First, we select an $i_1 \in I_1$ such that
  $\sigma(\emptyset) \cap U_{i_1}\ne \emptyset$, and put $\theta(\emptyset) =\sigma(\emptyset) \cap U_{i_1}$. Let $V_1
  \subseteq \theta(\emptyset)$ be $\alpha's$ first move in $Ch(T)$.
  Now, we choose $i_2 \in I_2$ such that $i_1 = \pi_1(i_2)$ and $\sigma(V_1) \cap U_{i_2} \ne \emptyset$. Define $\theta(V_1) = \sigma(V_1) \cap U_{i_2}$.
  Assume that for each $n> 1$, we have defined $\theta$ for any
  finite $\theta$-sequence $\langle V_1,...,V_n\rangle$ satisfying\\
  (1) $\theta(V_1,...,V_k)=\sigma(V_1,...,V_k)\cap
  U_{i_k}$ for all $k\le n$,\\
  (2) $i_k \in I_k$ for all $k\le n$,\\
  (3) $i_{k-1} = \pi_k(i_n)$ for all $2\le k\le n$.\\
  Note that any $\theta$-sequence of length $k \le n$ in $Ch(T)$
  is also a $\sigma$-sequence of the same length in ${\mathscr
  G}_S(D)$. Let $\langle V_1,...,V_n\rangle$ be any finite
  sequence of nonempty open subsets satisfying the conditions
  (1)--(3). Suppose that $V_{n+1} \subseteq \theta(V_1,...,V_n)$
  is $\alpha$'s  $(n+1)$-th move in $Ch(T)$. We choose $i_{n+1}
  \in I_{n+1}$ such that $\pi_{n+1}(i_{n+1})=i_n$ and $\sigma(V_1,
  ...,V_n)\cap U_{i_{n+1}}\ne\emptyset$, and then define
  \[
  \theta(V_1,
  ...,V_n, V_{n+1}) = \sigma(V_1,...,V_n) \cap U_{i_{n+1}}.
  \]
  This completes the definition of $\theta$ inductively.

  \medskip
  Since $T$ is a Baire space, by Theorem \ref{thm:krom}, $\theta$
  is not a winning strategy for $\beta$ in $Ch(T)$. It follows
  that there is an infinite $\theta$-sequence $\langle V_n: n\ge
  1 \rangle$ of $\alpha$'s moves such that $\bigcap_{n\ge 1} V_n
  \ne \emptyset$.
  Now, let $\langle x_n: n\ge 1 \rangle$ be a sequence in $T$
  with $x_n \in V_n \cap D$ for each $n\ge 1$. Then, by the
  construction of $\theta$, we see that $x_n \in U_{i_n}$ for
  all $n\ge 1$. Note that $\langle i_n: n\ge 1 \rangle$ is a
  $\pi$-sequence. Since $\bigcap_{n\ge 1} U_{i_n}\ne\emptyset$,
  then $\langle U_{i_n}: n \ge 1 \rangle$ must be a complete
  sequence. For each $n \ge 1$, put $F_n =\{x_k: k \ge n\}$.
  Then the filterbase $\{F_n: n\ge 1 \}$ is controlled by
  $\langle U_{i_n}: n \ge 1 \rangle$. This implies that
  $\{F_n: n\ge 1 \}$ has a cluster point in $T$, which is
  equivalent to that $\{x_n: n\ge 1\}$ has a cluster point in $T$.
  Note that $\langle V_n: n\ge 1 \rangle$ is an infinite
  $\sigma$-sequence in the game ${\mathscr G}_S(D)$ to witness
  that $\sigma$ is not a winning strategy for $\beta$. Thus,
  $T$ is a strongly Baire space.
  \end{proof}

  \begin{remark}
  The question that how strongly Baire spaces can be identified
  was considered firstly by Kenderov et al. in \cite{KKM:01} and
  more recently by Arhangel'skii et al. in
  \cite{shua-coban-petar:10}. Let $T$ and $Z$ be two topological
  spaces with $T \subseteq Z$. The space $T$ is said to have
  \emph{star separation in $Z$}, if there exists a sequence $\langle
  {\mathscr G}_n: n \ge 1\rangle$ of families of open subsets
  of $Z$ such that for every pair of points $x\in T$ and $z\in Z
  \setminus T$, there exists some $n\ge 1$ such that at least
  one of stars ${\rm St}(x, {\mathscr G}_n)$ and ${\rm St}(z,
  {\mathscr G}_n)$ is not empty and contains only one of the two
  points. In particular, $T$ is said to have \emph{countable
  separation} if in some compactification $bT$, $X$ has star
  separation by a sequence $\langle {\mathscr G}_n: n \ge 1\rangle$
  of open subsets of $bT$ such that for each $n\ge 1$, $\mathscr
  G_n$ contains only one member. It was observed in \cite{KKM:01}
  that every Baire space having countable separation is strongly
  Baire. Recently, it has been shown that every Baire space which
  has star separation in some compactification is $(\alpha,
  G_{C^{\widetilde{}}})$-favorable, and thus is a strongly Baire
  space. Note that a $(\alpha, G_{C^{\widetilde{}}})$-favorable
  space is called an \emph{monotonically $p$-space} by Cao and
  Piotrowski in \cite{cao-piotrowski:04} and moreover, each
  monotonically $p$-space is an almost monotonically $p$-space.
  \end{remark}

  \section{Metrizability revisited} \label{sec:metrizable}

  In this section, we study the metrizability of topologies
  ${\mathscr T}_{\mathscr B}$ and ${\mathscr T}_{\mathscr B}^s$ respectively, and improve some results in the literature.
  Beer and Levi \cite{beer-levi:09} characterized the
  metrizability of ${\mathscr T}_{\mathscr B}^s$ in terms of
  $\mathscr B$, as shown in the following result.

  \begin{theorem}[\cite{beer-levi:09}] \label{thm:beer09}
  Let $(X,d)$ be a metric space and $\mathscr B$ be a bornology
  on $X$ with a closed base $\mathscr S$. The following
  statements are equivalent.\\
  {\rm (1)} The bornology $\mathscr B$ has a countable base.\\
  {\rm (2)} For each metric space $(Y,\rho)$, ${\mathscr
  T}_{\mathscr B}^s$ on $(C(X,Y))$ is metrizable.\\
  {\rm (3)} For each metric space $(Y,\rho)$, ${\mathscr
  T}_{\mathscr B}^s$ on $(C(X,Y))$ is first countable.\\
  {\rm (4)} The space $(C(X),{\mathscr T}_{\mathscr B}^s)$ has
  a countable local base at some $f\in C(X)$.
  \end{theorem}

  In \cite{caserta-dimaio-hola:10}, Caserta et al. provided
  further characterizations for the metrizability of ${\mathscr T
  }_{\mathscr B}^s$ in terms of other countability type
  properties, e.g., the pointwise countable property.

  \medskip
  Next, we will provide a new characterization of the metrizability
  of ${\mathscr T }_{\mathscr B}$ and ${\mathscr T }_{\mathscr
  B}^s$ , which will help us to study completeness properties
  of these topologies. First, we need to introduce some notation.
  Let $D\subseteq T$ be a dense set of a topological space $T$.
  A point $x\in T$ is
  called a \emph{$q_D$-point} of $T$ if there exists a sequence
  $\langle U_n: n\ge 1 \rangle$ of open neighborhood of $x$
  such that if $x_n \in U_n \cap D$ for all $n\ge 1$, the
  sequence $\langle x_n: n\ge 1 \rangle$ has a cluster point in
  $T$. If $D=T$, a $q_D$-point is simply called a \emph{q-point}.
  If every point of $T$ is a $q$-point, then $T$ is called a
  \emph{q-space}. Obviously, every $q$-point is a $q_D$-point.

  \begin{proposition} \label{prop:existqd}
  Let $T$ be a topological
  space. If $T$ is strongly Baire or contains a dense subset which
  is a $q$-space itself, then $T$ contains at least one $q_D$-point
  for some dense subset $D \subseteq T$.
  \end{proposition}

  \begin{proof}
  If $D \subseteq T$ is a dense subset which is $q$-space
  itself. For each $x\in D$, let $\langle U_n: n \ge 1\rangle$ be
  a sequence of open neighborhoods of $x$ in the subspace $D$ that
  witnesses $x$ to be a $q$-point in $D$. For each $n\ge 1$, pick
  an open subset $\widetilde{U}_n$ in $T$ such that $U_n =
  \widetilde{U}_n\cap D$. Then, $\langle \widetilde{U}_n: n \ge
  1 \rangle$ witnesses $x$ to be a $q_D$-point in $T$.

  \medskip
  Let $D \subseteq T$ be a dense subset of $T$ such that $\beta$
  does not have a winning strategy in the game ${\mathscr
  G}_S(D)$ played in $T$. Let $t_0$ be any strategy of
  $\beta$. Since $t_0$ is not a winning strategy when $\beta$
  applies $t_0$ in the game ${\mathscr G}_S(D)$, there exists
  a sequence $\langle U_n: n\ge 1\rangle$ of nonempty subsets
  of $T$ such that $\bigcap_{n\ge 1} U_n \ne \emptyset$ and every
  sequence $\langle x_n: n\ge 1 \rangle$ with $x_n \in U_n \cap
  D$ for every $n \ge 1$ has a cluster point in $T$. It is easy
  to see that every point in $\bigcap_{n\ge 1} U_n$ is a
  $q_D$-point of $X$.
  \end{proof}

  In Proposition \ref{prop:existqd}, if $T$ is homogeneous,
  then every point of $T$ is a $q_D$-point for some dense
  subset $D\subseteq T$. However, the converse of Proposition \ref{prop:existqd} does not hold even for homogeneous
  first countable spaces. For example, the Sorgenfrey line
  $\mathbb S$ is a Baire and first countable paratopological
  group, but not strongly Baire, since it is not a topological
  group, by Theorem 2 in \cite{KKM:01}.

  \begin{theorem} \label{thm:strongmetric}
  Let $(X,d)$ be a metric space and $\mathscr B$ be a bornology
  on $X$ with a closed base $\mathscr S$. Then, ${\mathscr
  T}_{\mathscr B}^s$ on $C(X)$ is metrizable if and only if the
  space $(C(X),{\mathscr T}_{\mathscr B}^s)$ has at least one
  $q_D$-point for some dense subset $D \subseteq C(X)$.
  \end{theorem}

  \begin{proof}
  The necessity is trivial. To prove the sufficiency, let $f_\ast
  \in C(X)$ be a $q_D$-point of $(C(X), {\mathscr T}_{\mathscr B}^s)$
  for some dense subset $D \subseteq C(X)$. Let $\widetilde{D}= D-
  f_\ast$. Then $\widetilde{D}$ is dense in $(C(X),
  {\mathscr T}_{\mathscr B}^s)$ and the zero function $0_X$ is a $q_{\widetilde{D}}$-point
  of $(C(X),{\mathscr T}_{\mathscr B}^s)$. This means that there is a
  sequence $\langle U_n: n \ge 1 \rangle$ of open neighborhoods
  of $0_X$ in $(C(X), {\mathscr T}_{\mathscr B}^s)$ such that any
  sequence $\langle f_n: n\ge 1\rangle$ with $f_n \in U_n \cap
  \widetilde{D}$ for all $n\ge 1$ has a cluster point in $(C(X),
  {\mathscr T}_{\mathscr B}^s)$. Without loss of generality, we can
  assume that $U_n=[B_n; \varepsilon_n]^s(0_X)$, where $B_n \in
  \mathscr S$ and $\varepsilon_n>0$ for all $n\ge 1$.

  \medskip
  First, we claim that $X=\bigcup_{n\ge 1} B_n$. Assume the
  contrary, then there will be a point $x_0\in X \setminus
  \bigcup_{n\ge 1} B_n$. For each $n\ge 1$, let $\lambda_n =
  \frac{1}{2}d(x_0,B_n) >0$. Since $x_0 \not\in
  \overline{B_n^{\lambda_n}}$ for each $n\ge 1$, we can choose
  a continuous function $f_n: X \to \mathbb R$ such that
  $f_n(x_0)=n$ and $f_n(x) =0$ for every point
  $x\in B_n^{\lambda_n}$. Next, for each $n\ge 1$,
  we choose a function $g_n\in C(X)$ such that
  \[
  g_n \in [B_n; \varepsilon_n]^s(0_X) \cap \left[\{x_0\},
  \frac{1}{3}\right]^s(f_n) \cap \widetilde{D}
  \]
  Since $0_X$ is a $q_{\widetilde{D}}$-point, then the sequence
  $\langle g_n: n\ge 1\rangle$ must have a cluster point $f_0$
  in $(C(X),{\mathscr T}_{\mathscr B}^s)$. However, this is not
  possible, since
  \[
  n-\frac{1}{3} < g_n(x_0) < n+\frac{1}{3}
  \]
  for all $n\ge 1$. This is a contradiction. Thus, the claim is
  verified.

  \medskip
  Now, choose a sequence $\langle V_n: n\ge 1 \rangle$ of open
  neighborhoods of $0_X$ in $(C(X), {\mathscr T}_{\mathscr B}^s)$
  with $V_n=[C_n; \delta_n]^s(0_X)$ such that\\
  (1) $C_n \in \mathscr S$ for all $n\ge 1$,\\
  (2) $B_n \subseteq C_n \subseteq C_{n+1}$ for all $n\ge 1$,\\
  (3) $0< \delta_{n+1} < \delta_n < \varepsilon_n$ for all $n\ge 1$,\\
  (4) $\bigcap_{n \ge 1} V_n =\bigcap_{n \ge 1} \overline{V}_n
  =\{0_X\}$.\\
  Note that (1)--(4) can be done as follows: Define $O_{ij}=[B_i;
  \frac{1}{j}]^s(0_X)$ for all $i,j \ge 1$, and re-label as sets
  $G_n$ so that $\langle G_n: n\ge 1\rangle =\langle O_{i,j}: i, j
  \ge 1\rangle$. For any function $f\in C(X)\setminus \{0_X\}$, there
  is some point $y_0 \in X$ with $f(y_0)\ne 0$. Since $X=\bigcup_{n
  \ge 1} B_n$, we can select $i_0 \ge 1$ such that $y_0 \in B_{i_0}$.
  In addition, choose $j_0 \ge 1$ such that $|f(y_0)|>\frac{1}{j_0}$.
  It follows that $f\not\in O_{i_0j_0}$. Thus, $\bigcap_{n\ge 1} G_n
  =\{0_X\}$. Regularity of $(C(X),{\mathscr T}_{\mathscr B}^s)$ allows
  us to shrink the sets $G_n$ to open sets $H_n$ so that
  $0_X \in H_n \subseteq \overline{H}_n \subseteq G_n$.
  By shrinking further if necessary, we may assume that
  $V_n$ is of the required form.

  \medskip
  Finally, we claim that $\{C_n: n \ge 1\}$ is a base for $\mathscr B$.
  If not, there exists some element $B \in \mathscr B$ such that
  $B\not \subseteq C_n$ for any $n\ge 1$. For each $n \ge 1$, we first
  select a point $x_n \in B \not \in C_n$ and put $\eta_n=\frac{1}{2}
  d(x_n,C_n)>0$. Since $x_n \not \in \overline{C_n^{\eta_n}}$ for
  each $n\ge 1$, we can select a continuous function $h_n: X \to
  \mathbb R$ such that $h_n(x_n)=1$ and $h_n(x)=\{0\}$ for each
  point $x\in C_n^{\eta_n}$. Like what we have done previously,
  for each $n\ge 1$, we select a continuous function $p_n: X \to
  \mathbb R$ such that
  \[
  p_n \in [C_n; \varepsilon_n]^s(0_X) \cap \left[\{x_n\},
  \frac{1}{3}\right]^s(h_n) \cap \widetilde{D}.
  \]
  Since $0_X$ is a $q_{\widetilde{D}}$-point of $(C(X),{\mathscr
  T}_{\mathscr B}^s)$, the sequence $\langle p_n: n \ge 1\rangle$
  has a cluster point $p_\ast$ in $(C(X), {\mathscr T}_{\mathscr
  B}^s)$. By (4), we have
  $p_\ast =0_X$. However, since $p_n \in [\{x_n\}, \frac{1}{3}]^s
  (h_n)$ implies $\frac{2}{3}<p_n(x_n)<\frac{4}{3}$, we conclude
  that $p_n \not \in [B;\frac{1}{3}]^s(0_X)$ for any $n\ge 1$. This
  is a contradiction, which verifies that $\{C_n: n \ge 1\}$ is a
  base for $\mathscr S$. Therefore, by Theorem \ref{thm:beer09}, $(C(X),{\mathscr T}_{\mathscr B}^s)$ is metrizable.
  \end{proof}

  \begin{theorem} \label{thm:metric}
  Let $X$ be a Tychonoff space and $\mathscr B$ be a bornology
  on $X$ with a closed base $\mathscr S$. Then, ${\mathscr
  T}_{\mathscr B}$ on $C(X)$ is metrizable if and only if the
  space $(C(X),{\mathscr T}_{\mathscr B})$ has at least one
  $q_D$-point for some dense subset $D \subseteq C(X)$.
  \end{theorem}

  \begin{proof}
  The proof is is omitted, as it is similiar to that of
  Theorem~\ref{thm:strongmetric}.
  \end{proof}

  \begin{remark}
  Theorem \ref{thm:strongmetric} and Theorem \ref{thm:metric} can
  be re-stated in term of an arbitrary metric space $(Y,\rho)$ just
  like that in Theorem \ref{thm:beer09}.
  \end{remark}

  \section{Completeness properties} \label{sec:complete}

  In this section, we study completeness properties of topologies
  ${\mathscr T}_{\mathscr B}$ and ${\mathscr T}_{\mathscr B}^s$ on
  $C(X,Y)$ when $X$ is either a normal or metric space and $Y$
  is a complete metric space.
  Let $X$ be topological space and $\mathscr B$ be a bornology on
  $X$. We call $X$ a \emph{${\mathscr B}_{\mathbb R}$-space}
  provided that for any function $f: X \to \mathbb R$, $f$
  is continuous if and only if $f\uhr B: B \to \mathbb R$ is
  continuous for all $B \in \mathscr B$. If ${\mathscr B} =
  {\mathscr K}(X)$, a ${\mathscr B}_{\mathbb R}$-space is called
  a \emph{$k_{\mathbb R}$-space}, \cite{mccoy-natantu:86} and
  \cite{mccoy-natantu:88}.

  \begin{theorem} \label{thm:completemetric}
  Let $X$ be a normal space and $\mathscr B$ be a bornology on $X$
  with a closed base $\mathscr S$. Then, ${\mathscr T}_{\mathscr B}$
  on $C(X)$ is completely metrizable if and only if $(C(X),{\mathscr
  T}_{\mathscr B})$ is strongly Baire and $X$ is a ${\mathscr
  B}_{\mathbb R}$-space.
  \end{theorem}

  \begin{proof}
  Assume that ${\mathscr T}_{\mathscr B}$ on $C(X)$ is completely
  metrizable. It is obvious that $(C(X),{\mathscr T}_{\mathscr B})$
  is strongly Baire. To see that $X$ is a ${\mathscr B}_{\mathbb
  R}$-space, note that if ${\mathscr T}_{\mathscr B}$ on $C(X)$
  is completely metrizable, then $(C(X), \Delta_{\mathscr B})$
  is complete. Let $\mathscr S$ be endowed with the partial
  ordered $\preccurlyeq$ defined by the set inclusion. Let $f: X
  \to \mathbb R$ be a
  function such that $f\uhr A$ is continuous for all $A \in
  \mathscr S$. By normality of $X$, each $f\uhr A$ can be
  extended to a continuous function $f_A: X \to \mathbb R$.
  Then, it is easy to verify that  $\langle f_A, {\mathscr S},
  \preccurlyeq \rangle$ is a Cauchy net and converges to $f$ in
  $(C(X), \Delta_{\mathscr B})$. Thus, the function $f$ is
  continuous.

  \medskip
  Now, suppose that $(C(X),{\mathscr T}_{\mathscr B})$ is strongly
  Baire and $X$ is a ${\mathscr B}_{\mathbb R}$-space. First, by
  Proposition \ref{prop:existqd}, $(C(X),{\mathscr T}_{\mathscr B})$
  contains at least one $q_D$-point. Thus, Theorem \ref{thm:metric}
  implies that ${\mathscr T}_{\mathscr B}$ on $C(X)$ is metrizable.
  First, we need to show that ${\mathscr K}(X)\subseteq \mathscr B$.
  Take an element $A\in {\mathscr K}(X)$.
  For each $n\ge 1$, put
  \[
  F_n= \{f\in C(X): |f(x)| \le n \mbox{ for every } x \in A\}.
  \]
  It is clear that $G_n$ is closed and $C(X)=\bigcup_{n\ge 1} F_n$.
  Since $(C(X), {\mathscr T}_{\mathscr B})$ is strongly Baire, then
  there exists some $n_0 \geq 1$ such that ${\rm Int}(F_{n_0}) \ne
  \emptyset$. It follows that there exist an element $B\in \mathscr
  S$, $\varepsilon >0$ and $f\in C(X)$ such that $[B;\varepsilon]
  (f) \subseteq F_{n_0}$. We claim that $A \subseteq B$. If not,
  we can pick up a point $x_0 \in A\setminus B$, and then define a
  continuous function $g: B \cup \{x_0\} \to \mathbb R$ such that
  \[
  g(x) = \left \{
  \begin{array}{ll}
  f(x), & \mbox{if $x \in B$;}\\[0.5em]
  n+1, & \mbox{if $x=x_0$.}
  \end{array}
  \right.
  \]
  Since $g$ is continuous on a closed subspace $B \cup \{x_0\}$
  of a normal space $X$, then $g$ has an extension
  $\widetilde{g} \in C(X)$. It is clear that $\widetilde{g} \in
  [B;\varepsilon](f)$. However, $\widetilde{g}\not\in F_n$, as
  $|\widetilde{g}(x_0)|> n$. Thus, the claim is verified. It follows
  that $A\in {\mathscr B}$. Let $\{B_n: n \ge 1\}$ be a countable base
  of $\mathscr B$. Then, it is easy to see that $\{[B_k, \frac{1}{n}]:
  k, n \ge 1\}$ is a countable base for $\Delta_{\mathscr B}$. This
  implies that $\Delta_{\mathscr B}$ is metrizable. To show that
  $\Delta_{\mathscr B}$ is complete, let $\langle f_n: n\ge 1\rangle$
  be a $\Delta_{\mathscr B}$-Cauchy sequence in $C(X)$. For each
  $x\in X$, the sequence $\langle f_n(x): n\ge 1\rangle$ has a
  limit point $\lim_{n\ge 1}f_n(x)$. Define a function $f: X \to
  \mathbb R$ such that $f(x) = \lim_{n\ge 1}f_n(x)$. It can be
  verified readily that the sequence $\langle f_n\uhr B: n \ge 1
  \rangle$ converges uniformly to $f\uhr B$ for each $B\in \mathscr
  B$. It follows that $f\uhr B$ is continuous for all $B \in\mathscr
  B$. Since $X$ is a ${\mathscr B}_{\mathbb R}$-space, $f$ is
  continuous. Also, $\langle f_n: n\ge 1\rangle$ ${\mathscr
  T}_{\mathscr B}$-converges to $f$.
  \end{proof}

  \begin{corollary} \label{coro:completemetric}
  Let $(X,d)$ be a metric space and $\mathscr B$ be a bornology
  with a closed base $\mathscr S$. Then,${\mathscr T}_{\mathscr B}$
  on $C(X)$ is completely metrizable if and only if the space
  $(C(X), {\mathscr T}_{\mathscr B})$ is strongly Baire.
  \end{corollary}

  \begin{remark}
  The strong Baire property was firstly applied to the study of
  function space $C_k(X)$ when $X$ is a manifold in \cite{CGGM:08}.
  More precisely, it is shown in \cite{CGGM:08} when $X$ is a
  manifold, $C_k(X)$ is strongly Baire if and only if $M$ is
  metrizable.
  \end{remark}

  Recently, complete metrizability of ${\mathscr T}_{\mathscr B}^s$
  has been studied by Hol\'{a} and Caserta et al. To state their
  results, we need to induce the concept of a shield. Let $(X,d)$
  be a metric space and let $A$ be a nonempty subset of $X$. A
  superset $B$ of $A$ is called a \emph{shield} for $A$ if for any
  nonempty closed set $C$ with $B \cap C =\emptyset$ then $D(A,C)>0$,
  \cite{beer-costantini:13}. A bornology $\mathscr B$ on $(X,d)$
  is said to be shielded from closed sets if every $A\in \mathscr
  B$ has a shield in $\mathscr B$.

  \begin{theorem} [\cite{caserta-dimaio-hola:10}, \cite{caserta-dimaio-hola:12}] \label{thm:strongcomplete}
  Let $(X,d)$ be a metric space and $\mathscr B$ be a bornology
  with a closed base $\mathscr S$. Then the following statements
  are equivalent.\\
  {\rm (1)} ${\mathscr K}(X)\subseteq \mathscr B$ and $\mathscr B$
  is shielded from closed sets and has a countable base.\\
  {\rm (2)} For each complete metric space $(Y,\rho)$,
  ${\mathscr T}_{\mathscr B}^s$ on $C(X,Y)$ is completely metrizable.\\
  {\rm (3)} ${\mathscr T}_{\mathscr B}^s$ on $C(X)$ is
  completely metrizable.\\
  {\rm (4)} The space $(C(X), {\mathscr T}_{\mathscr B}^s)$ is
  \v{C}ech complete.
  \end{theorem}

  To see the role that the concept of a shield plays in the previous
  result, suppose that some $B_0 \in \mathscr B$ has no
  shield in $\mathscr B$. Hol\'{a} \cite{hola:12} constructs a
  Cauchy sequence $\langle f_n: n \ge 1 \rangle$ in $(C(X),
  \Delta_{\mathscr B}^s)$ which has no limit point. By applying
  this fact, one is able to show that if ${\mathscr T}_{\mathscr
  B}^s$ on $C(X)$ is completely metrizable, then $\mathscr B$ is shielded from closed sets. Note that if $\mathscr B$ is shielded
  from closed sets, then ${\mathscr T}_{\mathscr B}^s = {\mathscr T}_{\mathscr B}$ on $C(X)$. Thus, the following open question
  is of some interests.

  \begin{question}
  Let $(X,d)$ be a metric space and $\mathscr B$ be a bornology
  with a closed base $\mathscr S$. If $(C(X),{\mathscr
  T}_{\mathscr B}^s)$ is strongly Baire, must it be completely metrizable?
  \end{question}

   A topological space $T$ is called \emph{pseudo-complete}
   \cite{oxtoby:61} if it has a sequence $\langle {\mathscr B}_n:
  n \ge 1 \rangle$ of $\pi$-bases such that $\bigcap_{n\ge 1} V_n\ne
  \emptyset$ whenever $\overline{V}_{n+1} \subseteq V_n \in {\mathscr
  B}_n$ for each $n \ge 1$, and $T$ is called a \emph{$\sigma$-space} \cite{Gr:84} if it has a $\sigma$-discrete
  network.

  \begin{theorem} \label{thm:morecomplete}
  Let $(X,d)$ be a metric space and $\mathscr B$ be a bornology
  on $X$ with a closed base $\mathscr S$. The following
  statements are equivalent.\\
  {\rm (1)} ${\mathscr T}_{\mathscr B}$ on $C(X)$ is
  completely metrizable.\\
  {\rm (2)} The space $(C(X),{\mathscr T}_{\mathscr B})$ is sieve
  complete.\\
  {\rm (3)} The space $(C(X),{\mathscr T}_{\mathscr B})$ is almost
  complete.\\
  {\rm (4)} The space $(C(X),{\mathscr T}_{\mathscr B})$ is a
  pseudo-complete and $\sigma$-space.\\
  {\rm (5)} The space $(C(X),{\mathscr T}_{\mathscr B})$ is a
  pseudo-complete space with at least one $q_D$-point for some
  dense subset $D \subseteq C(X)$.
  \end{theorem}

  \begin{proof}
  (1) $\Rightarrow$ (4) is trivial.
  The equivalence of (1), (2) and (3) follows from Corollary
  \ref{coro:completemetric} and the fact that every almost complete
  space is strongly Baire.

  \medskip
  (4) $\Rightarrow$ (5). Each pseudo-complete space is Baire. It
  was shown in \cite{vanDouwen:77} that every Baire $\sigma$-space
  contains a dense metrizable subspace. Thus, if $(C(X),{\mathscr T}_{\mathscr B})$ is a pseudo-complete and $\sigma$-space, then
  it is a pseudo-complete space with at least one $q_D$-point for
  some dense subset $D \subseteq C(X)$.

  \medskip
  (5) $\Rightarrow$ (1). If $(C(X),{\mathscr T}_{\mathscr B})$ is a pseudo-complete space with with at least one $q_D$-point for some
  dense subset $D \subseteq C(X)$, by Theorem \ref{thm:metric},
  $(C(X),{\mathscr T}_{\mathscr B})$ is a metrizable and Baire
  space, and thus is strongly Baire. Thus, by Corollary
  \ref{coro:completemetric}, ${\mathscr T}_{\mathscr B}$ on $C(X)$
  is completely metrizable.
  \end{proof}

  \begin{theorem} \label{thm:strongalmost}
  Let $(X,d)$ be a metric space and $\mathscr B$ be a bornology
  with a closed base $\mathscr S$. The following statements
  are equivalent.\\
  {\rm (1)} The space $(C(X),{\mathscr T}_{\mathscr B}^s)$ is almost
  complete.\\
  {\rm (2)} The space $(C(X),{\mathscr T}_{\mathscr B}^s)$ is a
  pseudo-complete and $\sigma$-space.\\
  {\rm (3)} The space $(C(X),{\mathscr T}_{\mathscr B}^s)$ is a
  pseudo-complete space with at least one $q_D$-point for some
  dense subset $D \subseteq C(X)$.
  \end{theorem}

  \begin{proof}
  The proof of (2) $\Rightarrow$ (3) is similar to that of
  (4) $\Rightarrow$ (5) in Theorem \ref{thm:morecomplete}.

  \medskip
  (1) $\Rightarrow$ (2). Every Tychonoff almost complete space
  is pseudo-complete and has a \v{C}ech complete dense
  $G_\delta$-subspace. Thus, if $(C(X), {\mathscr T}_{\mathscr B}^s)$
  is almost complete, by Proposition \ref{prop:existqd} and Theorem
  \ref{thm:strongmetric}, it is pseudo-complete and metrizable.

  \medskip
  (3) $\Rightarrow$ (1). Assume (3) holds. By Theorem
  \ref{thm:strongmetric}, $(C(X), {\mathscr T}_{\mathscr B}^s)$
  is metrizable. It was shown in Corollary 2.4 of \cite{aarts-lutzer:73} that a metrizable space is
  pseudo-complete if and only if it has a dense completely
  metrizable subspace.
  \end{proof}

  \begin{remark}
  Theorem \ref{thm:completemetric}, Corollary
  \ref{coro:completemetric}, Theorem \ref{thm:morecomplete} and
  Theorem \ref{thm:strongalmost} can be re-stated in term of an
  arbitrary metric space $(Y,\rho)$ just like that in Theorem 3.1
  of \cite{hola:12}.
  \end{remark}

  \section{The Baire property}

  In this section, we provides some results on the Baire property
  of ${\mathscr T}_{\mathscr B}^s$. In general,
  it is not an easy task to characterize when a function space
  has the Baire property. The problem for $C_p(X)$ was solved
  independently by Pytkeev \cite{pytkeev:85}, Tkachuk
  \cite{tkachuk:85} and van Douwen \cite{vandouwen:94}, and the
  problem for $C_k(X)$ was solved for locally compact $X$ by
  Gruenhage and Ma \cite{gruenhage:97}.

  \medskip
  A subset $A$ of a topological space $X$ is \emph{relatively
  pseudocompact} if $f(A)$ is bounded in $\mathbb R$ for all $f\in
  C(X)$. For a metric space $(X,d)$ and a bornology $\mathscr B$
  on $X$ with a closed base $\mathscr S$, Hol\'{a} \cite{hola:12}
  showed that if $(C(X), {\mathscr T}_{\mathscr B}^s)$ is a Baire space, then every relatively pseudocompact subset of $X$ is
  in $\mathscr B$. Consequently, if $(C(X),
  {\mathscr T}_{\mathscr B}^s)$ is Baire, ${\mathscr T}_k
  \subseteq {\mathscr T}_{\mathscr B}^s$. A classical
  result in \cite{mccoy-natantu:88} states that for a first
  countable paracompact space $X$, $C_k(X)$ is Baire if and only
  if $X$ is locally compact. Next, we provides two examples to
  show that this result fails for functions spaces defined in
  terms of bornologies.

  \begin{example} \label{exam:non-baire}
  \emph{There exist a locally compact metric space $(X,d)$ and a
  bornology $\mathscr B$ on $(X,d)$ with a closed base such that
  $(C(X), {\mathscr T}_{\mathscr B})$ is not a Baire space.} Let
  $X = \bigcup_{\alpha < \omega_1} (0,1) \times \{\alpha\}$,
  where $(0, 1)$ is the open unit interval. For each $n\ge 1$,
  let $D_n \subseteq X$ be the subset defined by
  \[
  D_n= \left\{
  \frac{m}{2^n}: 0<m<2^n \right\} \times \omega_1.
  \]
  Define a metric $d$ on $X$ such that for any $(x,\alpha),
  (y, \beta) \in X$,
  \[
  d((x,\alpha), (y,\beta))=\left \{
  \begin{array}{ll}
  |x-y|, & \mbox{if $\alpha -\beta$;}\\[0.5em]
  2, & \mbox{if $\alpha \not= \beta$.}
  \end{array}
  \right.
  \]
  Obviously, $(X,d)$ is a locally compact metric space.
  Let ${\mathscr B}$ be the bornology on $X$ generated by
  ${\mathscr K}(X) \cup \{D_n: n \ge 1\}$.

  \medskip
  {\bf Claim.} \emph{The space $(C(X), {\mathscr T}_{\mathscr B})$
  is not  Baire}. We shall construct a winning strategy $\sigma$
  for player $\beta$ in the Choquet game played in $(C(X), {\mathscr
  T}_{\mathscr B})$. Without loss of generality, we assume that
  players $\beta$ and $\alpha$ will choose basic open sets of
  the form $[A; \varepsilon](f)$, where $A \subseteq X$ is the
  union of a compact set and some $D_n$, $n\in \mathbb N$. Let
  $\sigma(\emptyset)=U_1=[B_1; 1](g_1)$ with $B_1 \supseteq D_1$.
  At the $n$-th round with $n\ge 1$, $\beta$ will have played
  $U_i=[B_i; \delta_i](g_i)$ and $\alpha$ will have played
  $V_i = [A_i;\varepsilon_i](f_i)$ with $1\le i \le n$ such that
  $V_i \supseteq U_{i+1}
  \supseteq V_{i+1}$ for each $1\le i<n$. Choose $k_{n+1} \in
  \mathbb N$ such that $A_n \setminus D_{k_{n+1}-1}$ is contained
  in a compact set. Let $I_{n+1}$ be a finite subset of $\omega_1$
  such that $A_n \setminus ((0, 1)\times I_{n+1})$ is a subset
  of $D_{k_{n+1}-1}$.
  Let $C_{n+1}= \left(D_{k_{n+1}} \setminus D_{k_{n+1}-1}\right)
  \times (\omega_1 \setminus I_{n+1})$. Then $C_{n+1}$ and $A_n$
  are disjoint nonempty closed subsets of $X$, thus we can find a
  function $g_{n+1} \in C(X)$ such that $g_{n+1} \uhr_{A_n} =
  f_n\uhr_{A_n}$ and $g_{n+1}({C_{n+1}})=\{n\}$. Let
  $\delta_{n+1} < \frac{1}{2} \varepsilon_n$, and define
  \[
  \sigma(V_1,...,V_n)=U_{n+1} =[A_n\cup D_{k_{n+1}};\delta_{n+1}]
  (g_{n+1})
  \]
  Then $U_{n+1} \subseteq V_n$.

  \medskip
  Let $\langle V_n: n \ge 1 \rangle$ be any $\sigma$-sequence of
  player $\alpha$'s moves, where $V_n=[A_n; \varepsilon_n](f_n)$
  for each $n\in \mathbb N$. Let $\mu< \omega_1$ be such that $(0,1)
  \times \{\mu\}$ does not intersect any of the compact sets that
  appear in $A_n$ for each $n\in \mathbb N$. If there exists a
  continuous function $g \in \bigcap_{n\ge 1} \sigma(V_1,...,V_n)$,
  then $g\uhr_{\left[\frac{1}{4}, \frac{3}{4}\right] \times \{\mu\}}$
  would be a continuous unbounded function on a compact set. This
  is a contradiction, which implies $\bigcap_{n \ge 1} \sigma(U_1,
  ...,U_n) =\emptyset$. Hence, $\sigma$ is a winning strategy for
  $\beta$.
  \end{example}

  \begin{example}
  \emph{There exist a locally compact metric space $(X,d)$ and a
  bornology $\mathscr B$ on $(X,d)$ with a closed base such that
  $(C(X), {\mathscr T}_{\mathscr B}^s)$ is not a Baire space.}
  The space $(X,d)$ is exactly the same as the one defined in
  Example~\ref{exam:non-baire}. But, the bornology $\mathscr B$
  is different from the one in Example~\ref{exam:non-baire}.
  To describe it, we need to set some notation to work with the
  complement of the Cantor set. The first interval is denoted
  by $(a_\emptyset, b_\emptyset)$ and the length of the interval
  is $\frac{1}{3}$. In the second stage, we have two new intervals
  $(a_0, b_0)$ and $(a_1,b_1)$, where $b_0< a_\emptyset$ and
  $b_\emptyset < a_1$, each of length $\frac{1}{3^2}$. In the
  $n+1$ stage, there are $2^n$ many intervals with length
  $\frac{1}{3^{n+1}}$ and each $p\in \{0,1\}^n$ corresponds to
  an interval $(a_p, b_q)$. For each $p\in \{0,1\}^n$, the two
  intervals that are closer to $(a_p, b_p)$ in the $n+2$ stage
  are $(a_{p^\wedge 0}, b_{p^ \wedge 0})$ at its left and
  $(a_{p^{\wedge}1}, b_{p^{\wedge}1})$ at its right. This gives
  $2^{n+1}$ many intervals $(a_q,b_q)$ with length $\frac{1}
  {n+2}$ in the $n+2$ stage, where $q\in \{0,1\}^{n+1}$. For
  notational convenience, let $\{0,1\}^0=\{\emptyset\}$
  and the length $p\in\{0,1\}^n$ is denoted by $\|p\|$. Of
  course, $\|\emptyset\|=0$. For each $n\ge 1$ and $p\in
  \bigcup_{0\le m\leq n} \{0, 1\}^m$, define
  \[
  E_{n, p}= \left[a_p-\sum_{k\ge \|p\|}^n\frac{1}{3^{k+3}},\
  b_p+ \sum_{k\ge \|p\|}^n\frac{1}{3^{k+3}}
  \right].
  \]
  Note that $E_{n,p} \subset E_{m,p}$ when $\|p\|\le n<m$. Now,
  for each $n\ge 0$, put
  \[
  E_n= \bigcup \left\{ E_{n,p}: p \in \{0,1\}^m \mbox{ and }
  0\le m \le n \right\} \ \ \mbox{and}\ \ \
  D_n= E_n \times \omega_1.
  \]
  Let $\mathscr B$ be the bornology on $X$ generated by
  ${\mathscr K}(X) \cup \{D_n: n \ge 0\}$.

  \medskip
  {\bf Claim.} \emph{The space $(C(X), {\mathscr T}_{\mathscr B}^s)$
  is not  Baire}. We shall construct a winning strategy $\sigma$
  for player $\beta$ in $Ch(C(X), {\mathscr T}_{\mathscr B}^s)$.
  Without loss of generality, we assume that players $\beta$ and $\alpha$ will choose basic open sets of the form $[A; \varepsilon]^s(f)$, where $A \subseteq X$ is the union of a
  compact set and some $D_n$, $n\ge 0$. Let $\sigma(\emptyset)=
  U_1=[B_1; 1]^s(g_1)$ with $B_1 \supseteq D_1$. At the $n$-th
  round with $n\ge 1$, $\beta$ will have played $U_i=[B_i;
  \delta_i]^s(g_i)$ and $\alpha$ will have played $V_i = [A_i;
  \varepsilon_i]^s(f_i)$ with $1\le i \le n$ such that
  $V_i \supseteq U_{i+1} \supseteq V_{i+1}$ for each $1\le i<n$.
  Choose $k_{n+1} \in \mathbb N$ such that $A_n \setminus
  E_{k_{n+1}-1}$ is contained in a compact set. Let $I_{n+1}$
  be a finite subset of $\omega_1$ such that $A_n \setminus ((0, 1)\times I_{n+1})$ is a subset of $D_{k_{n+1}-1}$. Let
  \[
  C_{n+1}= \left(\bigcup \left\{E_{k_{n+1}, p}: p\in \{0,
  1\}^{k_{n+1}} \right\}\right)\times (\omega_1 \setminus I_{n+1}).
  \]
  Then $C_{n+1}$ and $A_n$
  are disjoint nonempty closed subsets of $X$, thus we can find a
  function $g_{n+1} \in C(X)$ such that $g_{n+1} \uhr_{A_n} =
  f_n\uhr_{A_n}$ and $g_{n+1}({C_{n+1}})=\{n\}$. Let
  $\delta_{n+1} < \frac{1}{2} \varepsilon_n$, and define
  \[
  \sigma(V_1,...,V_n)=U_{n+1} =[A_n\cup D_{k_{n+1}};\delta_{n+1}]^s
  (g_{n+1})
  \]
  Then $U_{n+1} \subseteq V_n$.

  \medskip
  Let $\langle V_n: n \ge 1 \rangle$ be any $\sigma$-sequence of
  player $\alpha$'s moves, where $V_n=[A_n; \varepsilon_n]^s(f_n)$
  for each $n\in \mathbb N$. Let $\mu< \omega_1$ be such that $(0,1)
  \times \{\mu\}$ does not intersect any of the compact sets that
  appear in $A_n$ for each $n\in \mathbb N$. If there exists a
  continuous function $g \in \bigcap_{n\ge 1} \sigma(V_1,...,V_n)$,
  then $g\uhr_{\left[\frac{1}{9}, \frac{8}{9}\right] \times \{\mu\}}$
  would be a continuous unbounded function on a compact set. This
  is a contradiction, which implies $\bigcap_{n \ge 1} \sigma(U_1,
  ...,U_n) =\emptyset$. Hence, $\sigma$ is a winning strategy for
  $\beta$.
  \end{example}

  Let $\mathscr B$ be bornology on a metric space $(X,d)$. We consider
  the game ${\mathscr G}_{\mathscr B}^s(X)$ played on $(X,d)$
  between two players $\beta$ and $\alpha$. Players $\beta$ and
  $\alpha$ take turn to choose elements from $\mathscr B$ as follows:
  $\beta$ starts first and chooses $B_1 \in {\mathscr B}$, and $\alpha$
  responds by choosing $A_1\in \mathscr B$. At stage $n$ ($n\ge 2$),
  $\beta$ chooses a member $B_n \in \mathscr B$ which is far from
  $A_{n-1}$ and $\alpha$ responds by choosing another member $A_n\in
  \mathscr B$. The game is played continuously in this pattern and a
  play $\langle (B_n, A_n): n\ge 1\rangle$ is produced. We say that
  player $\alpha$ \emph{wins} the play $\langle (B_n, A_n): n\ge 1
  \rangle$ if the sequence $\langle B_n: n\ge 1 \rangle$ has an open
  discrete expansion in $X$. Otherwise, $\beta$ is said to have won
  the play.

  \begin{theorem} \label{thm:baire}
  Let $(X,d)$ be a metric space and let $\mathscr B$ be a bornology
  on $X$ with a closed base. If $(C(X), {\mathscr T}_{\mathscr B}^s)$
  is a Baire space, $\beta$ does not have a winning strategy in
  ${\mathscr G}_{\mathscr B}^s(X)$.
  \end{theorem}

  \begin{proof}
  Suppose that $\sigma$ is a strategy for $\beta$ in ${\mathscr
  G}_{\mathscr B}^s(X)$. We shall show that $\sigma$ is not a
  winning strategy for $\beta$. To achieve this, we define
  inductively a strategy $\tau$ for player $\beta$ in the
  Choquet game $Ch((C(X),{\mathscr T}_{\mathscr B}^s))$ and then
  apply the Baireness of $(C(X),{\mathscr T}_{\mathscr B}^s)$.

  \medskip
  First, let $f_1 = 1$ be the constant function and define
  $\tau(\emptyset):=\left[\sigma(\emptyset); \frac{1}{2}\right]^s
  (f_1)$. Suppose that $V_1 \subseteq \tau(\emptyset)$ is player
  $\alpha$'s first move in $Ch((C(X), {\mathscr T}_{\mathscr B}^s))$.
  Then, there exists a basic open set $U_1=[A_1; \varepsilon_1]^s
  (g_1)\subseteq V_1$, where $A_1\in {\mathscr S}$, $\varepsilon>0$
  and $g_1\in C(X)$. Since $U_1\subseteq \tau(\emptyset)$, without
  loss of generality, we can require $0 \le \varepsilon_1 \le
  \frac{1}{2}$ and $\sigma(\emptyset)\subseteq A_1$. Furthermore,
  since $\sigma(A_1) \cap \overline{A_1^{\delta_1}} =\emptyset$
  for some $\delta_1 >0$, by the Tietz extension theorem, there
  exists a function $f_2 \in C(X)$ such that $f_2\uhr_{A_1^{\delta_1}}
  = g_1 \uhr_{A_1^{\delta_1}}$ and $f_2(\sigma(A_1))=\{2\}$. Next,
  we define
  \[
  \tau(V_1) =\left[A_1 \cup \sigma(A_1); \frac{1}{2}\varepsilon_1
  \right]^s(f_2).
  \]
  Then $\tau(V_1) \subseteq U_1 \subseteq V_1 \subseteq
  \tau(\emptyset)$.

  \medskip
  Suppose that we have constructed $\tau$ for all finite legal
  moves of $\alpha$ with length $n$ ($n\ge 1$) in $Ch((C(X),
  {\mathscr T}_{\mathscr B}^s))$. In particular, for each sequence
  $\langle V_1,...,V_n\rangle$ of $\alpha$'s legal moves, there
  are finite sequences $\langle A_1,...,A_n \rangle$ in $\mathscr
  S$, $\langle g_1, ..., g_n\rangle$ and
  $\langle f_1, ..., f_{n+1}\rangle$ in  $C(X)$, $\langle
  \varepsilon_1, ..., \varepsilon_n \rangle$ and $\langle
  \delta_1, ..., \delta_n \rangle$ of positive real numbers, and
  $\langle U_1,...,U_n \rangle$ in ${\mathscr T}_{\mathscr B}^s$
  such that\\
  (1) $\displaystyle 0<\varepsilon_k \le \frac{1}{2^k}$ for all
  $1\le k\le n$;\\
  (2) $\langle A_k: 1 \le k \le n\rangle$ is a $\sigma$-sequence
  with $\sigma(A_1,...,A_{k-1})\subseteq A_k$ for all $1\le k\le
  n$;\\
  (3) $U_k = [A_k; \varepsilon_k]^s(g_k)$ for all
  $1\le k\le n$;\\
  (4) $U_k \subseteq V_k$ for all $1\le k\le n$;\\
  (5) $f_{k+1} \uhr_{A_k^{\delta_k}} = g_k \uhr_{A_k^{\delta_k}}$
  and $f_{k+1}(\sigma(A_1,...,A_k))=\{k+1\}$ for all $1\le k \le n$;\\
  (6) $\displaystyle \tau(U_1,...,U_k)=\left[A_k \cup \sigma(A_1,
  ...,A_k); \frac{1}{2}\varepsilon_k \right]^s (f_{k+1})$ for all
  $1\le k\le n$.\\
  Note that (3)--(6) imply that\\
  (7) $\tau(V_1,...,V_k) \subseteq U_k \subseteq V_k$ for all
  $1\le k\le n$.\\
  Now, let $\langle V_1,...,V_n, V_{n+1} \rangle$ be a finite
  sequence of legal moves of $\alpha$ with length $n+1$ in
  $Ch((C(X),{\mathscr T}_{\mathscr B}^s))$. We choose a basic
  open set $U_{n+1}$ in ${\mathscr T}_{\mathscr B}^s$
  such that
  \[
  U_{n+1}=[A_{n+1}; \varepsilon_{n+1}]^s(g_{n+1}) \subseteq V_{n+1},
  \]
  where $A_{n+1} \in \mathscr S$ such that $A_n \cup \sigma(A_1,
  ...,A_n)\subseteq A_{n+1}$ and $0<\varepsilon_{n+1}\le\frac{1}{2}
  \varepsilon_n$. Since $\sigma(A_1,...,A_{n+1}) \cap
  \overline{A_{n+1}^{\delta_{n+1}}} = \emptyset$ for some
  $\delta_{n+1}>0$, by the Tietz extension theorem, there exists
  a function $f_{n+1} \in C(X)$ such that $f_{n+1} \uhr_{A_{n+1}^{\delta_{n+1}}} = g_{n+1}
  \uhr_{A_{n+1}^{\delta_{n+1}}}$ and $f_{n+2}(\sigma(A_1,
  ...,A_{n+1}))=\{n+2\}$. Then, we define
  \[
  \tau(V_1,...,V_n,V_{n+1}) =\left[A_{n+1}\cup\sigma(A_1,...,
  A_{n+1}); \frac{1}{2} \varepsilon_{n+1}\right]^s(f_{n+2}).
  \]
  Then, by construction, we have that $\tau(V_1,...,V_n,V_{n+1})
  \subseteq U_{n+1}\subseteq V_{n+1}$. Thus, we have defined
  inductively the strategy $\tau$ for $\beta$ in $Ch(C(X),
  {\mathscr T}_{\mathscr B}^s)$.

  \medskip
  Since $(C(X), {\mathscr T}_{\mathscr B}^s)$ is a Baire space,
  by Theorem \ref{thm:krom}, $\tau$ is not a winning strategy
  for $\beta$ in $Ch(C(X), {\mathscr T}_{\mathscr B}^s)$.
  Therefore, there exists an infinite sequence $\langle V_n:
  n \ge 1 \rangle$ of $\alpha$'s legal moves such that
  \[
  \bigcap_{n \ge 1} V_n = \bigcap_{n\ge 1} \tau(V_1,...,V_n)
  \ne \emptyset.
  \]
  Let $f \in \bigcap_{n \ge 1} V_n$. Furthermore, let $\langle
  A_n: n\ge 1\rangle$, $\langle g_n: n\ge 1\rangle$, $\langle
  f_n: n\ge 1\rangle$, $\langle \varepsilon_n: n\ge 1 \rangle$,
  $\langle \delta_n: n\ge 1\rangle$ and $\langle U_n: n\ge 1
  \rangle$ be associated sequences such that for each $n\ge 1$,
  (1)--(6) are satisfied. Then (7) is also satisfied for all
  $k\ge 1$, which implies that $f \in U_n$ for all $n\ge 1$.
  Thus, for each $n\ge 1$, we have $|f(x) -f_{n}(x)| <\frac{1}
  {2^{n}}$ for all $x \in A_{n-1} \cup \sigma(A_1,...,A_{n-1})$.
  Since $f(\sigma(A_1,...,A_{n-1})) =\{n\}$, then
  \[
  n -\frac{1}{2^n} < f(x) < n +\frac{1}{2^n}
  \]
  for all $x \in \sigma(A_1,...,A_{n-1})$. It follows that
  $\{f(\sigma(\emptyset))\} \cup \{f(\sigma(A_1,...,A_n)):
  n\ge 1\}$ is a discrete family on $\mathbb R$, which has an
  open discrete expansion in $\mathbb R$. Since $f$ is continuous, $\{\sigma(\emptyset)\} \cup\{\sigma(A_1,...,A_n): n\ge 1\}$
  has an open discrete expansion in $X$. This implies that
  $\sigma$ is not a winning strategy for $\beta$ in ${
  \mathscr G}_{\mathscr B}^s(X)$.
  \end{proof}

  Given two collections $\mathscr C$ and $\mathscr D$ of nonempty
  subsets of a topological space $X$, $\mathscr C$ is called a
  \emph{moving off collection over $\mathscr D$} if, for
  each $D \in \mathscr D$, there exists $C \in \mathscr C$
  with $C \cap D = \emptyset$. The space $X$ is said to have
  the \emph{$(\mathscr C, \mathscr D)$-moving off property} if
  any subset ${\mathscr A} \subseteq {\mathscr C}$, which is a
  moving off collection over $\mathscr D$, contains an infinite
  subcollection which has a discrete open expansion in $X$.

  \begin{proposition} \label{prop:movingoff}
  Let $(X,d)$ be a metric space and let $\mathscr B$ be a
  bornology on $X$. If player $\beta$ does not have a winning
  strategy in the ${\mathscr G}_{\mathscr B}^s(X)$-game, then
  $(X,d)$ has the $(\bigcup_{\delta>0}{\mathscr B}^\delta,
  \mathscr B)$-moving off property.
  \end{proposition}

  \begin{proof}
  Suppose that $(X,d)$ does not have the $(\bigcup_{\delta>0}
  {\mathscr B}^\delta, \mathscr B)$-moving off property. Then,
  there exists a moving off collection ${\mathscr A} \subseteq
  \bigcup_{\delta >0} {\mathscr B}^\delta$ no infinite subset
  of which has a discrete open expansion. Now, we can construct
  a winning strategy $\sigma$ for $\beta$ in ${\mathscr
  G}_{\mathscr B}^s(X)$. Let $\sigma(\emptyset)$
  be an arbitrary member of $\mathscr B$. Assume that we have
  defined $\sigma$ for player $\alpha$'s all finite legal moves
  $\langle A_1,...,A_n \rangle$ of length $n\ge 1$. We choose
  $B_{n+1} \in \mathscr B$ and $\delta>0$ such that $B_{n+1}^\delta
  \in \mathscr A$ and $B_{n+1}^\delta \cap A_n =\emptyset$, and
  define $\sigma(A_1,...,A_n) =B_{n+1}$. Since $\{ \sigma(A_1,
  ...,A_n): n\ge 1 \}$ does not have a discrete open expansion,
  $\sigma$ is a winning strategy for $\beta$. This is a
  contradiction.
  \end{proof}

  By results \cite{gruenhage:97}, if a metric space $(X,d)$ has
  the $({\mathscr K}(X),{\mathscr K}(X))$-moving off property,
  then $(X,d)$ must be locally compact, and $C_k(X)$ is a
  Baire space if and only if $(X,d)$ has the $({\mathscr K}(X),
  {\mathscr K}(X))$-moving off property. Note that $(X,d)$ has
  the $({\mathscr K}(X), {\mathscr K}(X))$-moving off property
  if and only if it has the $(\bigcup_{\delta>0}{\mathscr K}
  (X)^\delta, {\mathscr K}(X))$-moving off property. In the light
  of Theorem \ref{thm:baire} and Proposition \ref{prop:movingoff},
  we conclude this paper with the following open question.

  \begin{question}
  Let $(X,d)$ be a metric space and let $\mathscr B$ be a bornology
  on $X$ with a closed base. If $(X,d)$ has the $(\bigcup_{\delta>0}
  {\mathscr B}^\delta, \mathscr B)$-moving off property, must
  $(C(X), {\mathscr T}_{\mathscr B}^s)$ be a Baire space?
  \end{question}

  \bigskip
  \centerline{\sc Acknowledgement}

  \bigskip
  The paper was written while the first author was Visiting
  Professor at the Universidade de S\~{a}o Paulo, financially
  supported by the Funda\c{c}\~{a}o de Amparo \`{a} Pesquisa do
  Estado de S\~{a}o Paulo under the grant number FAPESP 2012/25401-5,
  when he was on sabbatical leave in May 2013. He would
  like to express his gratitude to the Institute of Mathematics and
  Statistics at the University of S\~{a}o Paulo for the hospitality.
  The second author's research was supported by the grants Auxilio
  Regular PROC FAPESP 2012/01490-9 and Produtividade em Pesquisa
  CNPq 305612/2010-7.

  \end{document}